\documentclass[11pt]{article}
\usepackage{amsmath,amsthm,amsfonts,amssymb,amscd,latexsym}
\usepackage[latin1]{inputenc}
\usepackage[all]{xy}

\usepackage{authblk}

\newtheorem{theorem}{Theorem}[section]
\newtheorem{lemma}{Lemma}[section]
\newtheorem{corollary}{Corollary}[section]
\newtheorem{proposition}{Proposition}[section]

\newtheorem{definition}{Definition}[section]

\theoremstyle{plain}

\newcommand{\NN}{\mathbb N}
\newcommand{\ZZ}{\mathbb Z}

\newcommand{\QQ}{\mathbb Q}
\newcommand{\lex}{\stackrel{\mbox{\tiny{\emph{lex}}}}{\times}}

\title{Tropicalization through the lens of {\L}ukasiewicz logic, with a topos theoretic perspective}
\author[$\star$]{A. Di Nola}
\author[$\star$]{G. Lenzi}
\author[$\dag$]{B. Gerla}

\affil[$\star$]{Dep. Mathematics, University of Salerno (Italy) }
\affil[ ]{\tt{\{adinola, gilenzi\}@unisa.it}}
\affil[$\dag$]{Dep. Theoretical and Applied Sciences, University of Insubria (Italy)}
\affil[ ]{\tt{brunella.gerla@uninsubria.it}}

\date{}
\begin{document}

\maketitle

\begin{abstract}
The main aim of this paper is to show the interconnections between {\L}ukasiewicz logic and algebraic geometry using algebraic, geometric and logical instruments.
We continue our investigation into a new algebraic geometry based on idempotent semifields, in particular those related with MV-algebras, describing categorical equivalence between different structures related to tropical geometry and many valued logics. Further, we describe such connections in terms of topoi.
\end{abstract}

\section{Introduction}

The main aim of this paper is to combine the ideas of
\cite{CDM},  \cite{DL94}, \cite{R} and \cite{DG05} to show that the topics of {\L}ukasiewicz logic, tropical structures and algebraic geometry fruitfully meet. This gives rise to a topos theoretic perspective to {\L}ukasiewicz logic, see \cite{CC}. Our aim seems to be completely in line with the spirit of the following remark from \cite{M08}:\vskip 5pt

"$\emph{The confluence of geometric, combinatorial and}$

$\emph{logical-algebraic techniques on a common problem}$

$\emph{is one of the manifestations of the unity of mathematics}$."\vskip 5pt

Tropicalization, originally, is a method aimed at simplifying algebraic geometry and it is used in many applications. It can be seen as a part of idempotent algebraic geometry. Tropical geometry produces objects which are combinatorially similar to algebraic varieties, but piecewise linear. That is, tropicalization attaches a polyhedral complex to an algebraic variety,
obtaining a kind of algebraic geometry over idempotent semifields.

Originally the idea was focused on complex algebraic geometry and started from a field (or ring) $K$ and its polynomials, an idempotent semiring
$(S,\wedge,+,0,1)$ and a valuation $v:K\to S$.
A semiring is an abelian monoid $(A,+,0)$ with a multiplicative monoid structure $(A,\cdot,1)$ satisfying the distributive laws and such that $a \cdot 0 = 0 \cdot a = 0$ for all $a \in A$. A semiring is idempotent if $+$ is idempotent.

Usually $S=G\cup\{\infty\}$ where $G$ is a totally ordered abelian group and $\infty$ is an infinite extra element, but we suggest to be more liberally inspired by \cite{R}: first, $G$ can be any lattice ordered abelian group, and $S=G\cup\{\infty\}$ is a semifield (the $\emph{Rump semifield}$ of $G$) . Moreover, we point out that the passage from $G$ to $S$ and conversely is functorial, even an equivalence. So we propose here a $\emph{functorial tropicalization}$.

In this paper we intend to contribute to what we mean by tropical mathematics, which for us, in a very broad sense, is the use of idempotent semirings in mathematics.

The most used semiring in tropicalization is the tropical semiring $$\mathbb{R}^{maxplus}= (\mathbb{R}\cup\{-\infty\},max,+,-\infty,0)$$ (or its dual with min instead of max).
The variety generated by $\mathbb{R}^{maxplus}$ includes properly the semiring reduct of $[0,1]$. On the other hand, if we consider the negative cone of the reals, $\mathbb{R}^{neg}=(\mathbb{R}^{\leq 0}\cup\{-\infty\},max,+,-\infty,0)$, we have that the varieties generated by $[0,1]$ and $\mathbb{R}^{neg}$ are the same. Another famous semiring is the tropical semiring $\mathbb{N}^{trop}=(\mathbb{N}\cup\{+\infty\},min,+,+\infty,0)$. Once again, the varieties generated by $\mathbb{N}^{trop}$ and $[0,1]$ are the same. However, $\mathbb{N}^{trop}$ and $[0,1]$ have different first order theories. Thus, looking at these two varieties from a logical point of view, some differences emerge. This justifies approaching tropical algebraic structures from a logical point of view.

One can try to apply the same tropicalization idea to more general frameworks, like universal algebraic geometry by Plotkin \cite{P}. This paper proposes a possible generalization of tropicalization of algebraic structures
using many valued logic and especially fuzzy logic. In fact, the algebraic structures of fuzzy {\L}ukasiewicz logic (MV-algebras) have a natural structure of idempotent semiring.

{\L}ukasiewicz logic is a fuzzy logic, that is a logic where the set of values is the real interval $[0,1]$, rather than the case $\{0,1\}$ of classical logic.
The connectives of {\L}ukasiewicz logic are $x\oplus y=min(x+y,1)$ (replacing OR) and $\neg x=1-x$ (replacing NOT). Then AND gets replaced by $x\odot y=max(0,x+y-1)$ (the {\L}ukasiewicz product). We have the tertium non datur, $x\oplus \neg x=1$, the non contradiction law $x\odot \neg x=0$, but not the idempotency: $x\oplus x\not=x$. Despite the lack of idempotency,  MV-algebras are deeply connected with idempotent structures.

{\L}ukasiewicz logic can be axiomatized as follows, where $x\to y$ means $\neg x\oplus y$:

\begin{enumerate} \item $x\to(y\to x)$
\item $(x\to y)\to((y\to z)\to (x\to z))$
\item $((x\to y)\to y))\to ((y\to x)\to x)$
\item $(\neg x\to\neg y)\to(y\to x)$.
\end{enumerate}
The only rule is modus ponens: from $x\to y$ and $x$ derive $y$.

Like the semantic counterpart of classical logic is given by Boolean algebras,
the semantic counterpart of {\L}ukasiewicz logic is given by MV-algebras.

We focus on algebraic models of an equational extension of {\L}ukasiewicz logic. Actually we consider the extension of {\L}ukasiewicz logic given by the equation in the section 3.1 that generates the variety containing all {\em perfect MV-algebras}. Perfect MV-algebras are characterized by the fact that their elements are either infinitesimals or negation of infinitesimals.

 The interplay between semirings and MV-algebras is very interesting. Every MV-algebra is in a natural way (better, in two natural dual ways) a semiring, as explained in \cite{IEEE17}, \cite{DG05}.
Indeed, any MV-algebra $A$ has two semiring reducts that are $(A,0,1,\wedge,\oplus)$ and $(A,0,1,$ $\vee,$ $\odot)$. They are isomorphic semirings and an isomorphism is given by the negation (see \cite{DG05}).
Starting from such observation, we consider idempotent structures built from MV-algebras and from that we develop a new kind of algebraic geometry.In this paper we are inspired by universal algebraic geometry, a research program originated in \cite{P}. In particular, we apply algebraic geometry to MV algebras (see also \cite{JSL}) and to the interesting variety $V(C)$, generated by Chang MV algebra $C$. Intuitively, $V(C)$ algebras are Boolean algebras perturbed with infinitesimals. We think that the idea of generalizing algebraic geometry from rings and fields to varieties of universal algebra, instead of being a useless exercise, has been the key to solving difficult problems, like the conjectures  of Tarski for free groups:

{\bf Tarski Conjecture 1} Any two non-abelian free groups are elementarily equivalent. That is any two non- abelian free groups satisfy exactly the same first-order theory.

{\bf Tarski Conjecture 2} If the non-abelian free group H is a free factor in the free group G then the inclusion map  is an elementary embedding, see \cite{FGR}.

We hence organized the rest of the paper as follows:

In section $2$ we discuss the relations between MV-algebras, semirings and semifields, showing that the category of MV-algebras is equivalent to semiring based categories.

In section $3$ we use the variety $V(C)$ of MV algebras generated by $C$ to give a functorial notion of  tropicalization of an algebraic structure. We also present a geometric invariant $\theta$ for algebras from the variety $V(C)$ which has role in the definition of tropicalization functor. We further analyze the algebraic aspects of $\theta$ invariants, showing that they can be endowed with a rich algebraic structure.

In section $4$ we will show how a logic admitting only truth values infinitesimally close to boolean values, surprisingly, it can exhibit models that are functorially connected to a category of points of a non-commutative geometry

In the last section we sketch some conclusions.

\section{MV-algebras, semirings and semifields}\label{sect:semi}
MV-algebras are structures $(A,\oplus,0,1,\neg)$  where $(A,\oplus,0)$ is a commutative monoid, $\neg 0=1$, $x\oplus 1=1$, $\neg\neg x=x$  and $\neg(\neg x\oplus y)\oplus y)=\neg(\neg y\oplus x)\oplus x)$.

MV-algebras are also lattices under the ordering $x\leq y$ such that there is $z$ with $y=x\oplus z$. We let also $x\odot y=\neg(\neg x\oplus\neg y)$ and $x \ominus y=x \odot \neg y= \neg (\neg x \oplus y)$.

The main example of MV-algebra is $[0,1]$ with $x\oplus y=min(x+y,1)$ and $\neg x=1-x$. By \cite{D91} it follows that $[0,1]$ generates the variety of MV-algebras. Further, interpretation of formulas of {\L}ukasiewicz logic in the MV-algebra $[0,1]$ give a complete and sound semantics for the logic.

In order to define other interesting MV-algebras it is convienient to introduce Mundici functor $\Gamma$, see \cite{M86}. $\Gamma$ is a
 functorial equivalence between the category of MV-algebras and the category of Abelian $\ell$-groups with {\em strong unit}, i.e. an element $u \in G$, $u \geq 0$ such that
 for every element $g \in G$
 there exists a natural number
$n\geq 1$ such that
$g \leq nu$.

This equivalence was originally  used in order to clarify the relations between $AF$ $C^*$-algebras and their $K_0$ group. For any $\ell$-group $G$ with strong unit $u$, we get an MV-algebra $\Gamma(G,u)$ as the interval $[0,u]$ of $G$ where $x\oplus y=(x+y)\wedge u$ and $\neg x=u-x$.

 In case $G$ is the lexicographic product of two copies of the additive group of integer numbers, we get {\em Chang MV-algebra} $C=\Gamma(\ZZ\, \lex\,$ $\ZZ,$ $(1,$ $0))$.

Let us define {\em infinitesimal} an element $x$ such that $nx\leq \neg x$ for every positive integer $n$. The idea is that $x\leq 1/n$, but $1/n$ does not necessarily exist in an MV-algebra. The set of all infinitesimals of an MV-algebra $A$ is called the radical of $A$ and is denoted by $Rad(A)$. A perfect MV-algebra is an MV-algebra generated by its radical, hence by its infinitesimals. So, every perfect MV-algebra $A$ is the disjoint union of its radical and its co-radical, i.e., the set of negations of elements of the radical: $A=Rad(A) \cup co-Rad(A)$.

Note that $[0,1]$ has no infinitesimals, whereas $C$ has them. Indeed, $C=\{(0,n),(1,-m) \mid n,m \in \NN\}$ consists only of infinitesimals $(0,n)$ and negations of infinitesimals $(1,-m)$, that is, is perfect.
Every MV-algebra has a largest perfect subalgebra, consisting of the radical and coradical.

MV-algebras form a variety having countably many subvarieties. We will focus on the subvariety generated by $C$, denoted by $V(C)$. This variety is axiomatized by
\begin{center}
$(2x)^2=2(x^2)$
\end{center}
where $2x$ means $x\oplus x$ and $x^2$ means $x\odot x$.
Every perfect MV-algebra belongs to $V(C)$, but in $V(C)$ there are also MV-algebras that are not perfect, for example to $C \times C$.

In general, every MV-algebra $A$  has a Boolean part $B(A)$, a Boolean algebra consisting of the elements $x$ such that $x=x\oplus x$, and a perfect part $P(A)$, the perfect MV-algebra  generated by the infinitesimals of $A$. If $A\in V(C)$, then $A$ is a combination of $P(A)$ and $B(A)$, a perfect MV-algebra and a Boolean algebra. We can hence say that in $V(C)$-algebras every element is infinitely close to a Boolean element, that is to an element $x=x\oplus x$. So, we can consider $V(C)$-algebras as perturbed Boolean algebras (\cite{DLV22}).

Besides the categorical equivalence $\Gamma$, in   \cite {DL94} an equivalence $\Delta$ is established between abelian $\ell$-groups and perfect MV-algebras. Namely, if $G$ is an abelian $\ell$-group, then $\Delta(G)=\Gamma(\ZZ\ \lex\ G,(1,0))$. So, $C$ can also be defined as $\Delta(\ZZ)$.
By the functor $\Delta$, a perfect MV-algebra corresponds to an Abelian $\ell$-group. So in a sense, an MV-algebra in V(C) is a combination of a Boolean algebra and an Abelian $\ell$-group.  We can say that the Boolean algebra carries a qualitative information, and the $\ell$-group carries a quantitative information.


In the introduction,  we mentioned the role of semirings and idempotent semifields in the tropicalization issues. Now we observe that studies concerning the relationship between semirings and MV-algebras started in \cite{DG05} and developed in subsequent papers, see (\cite{IEEE17},\cite{DNR}, \cite{DNR2}, \cite{GB}).

Recall that in every additively idempotent semiring $(S,+,\cdot,0,1)$ there is a natural partial order such that $s\leq t$ if and only if $s+t=t$. Homomorphisms of additively idempotent semirings are increasing.

Now let us turn to semifields, which are our tropical objects.
A {\em semifield} is a semiring $(F,+,\times,0,1)$ where the  multiplication is commutative and every nonzero element has an inverse with respect to the multiplication.
A semifield is idempotent if $x+x=x$ for every $x\in F$.

We have several categorial results relating MV-algebras and semifields. First of all we have:
\begin{theorem}\label{th:idemsemifields}\cite{R}\label{thm:rump} There is an equivalence $\rho$ between the category of idempotent semifields and the category of
abelian $\ell$-groups.
\end{theorem}
The equivalence in Theorem \ref{th:idemsemifields} acts by associating with every idempotent semifield $(F,+,\times,0,1)$ the abelian $\ell$-group $(F\setminus \{0\},\cdot, 1)$.

 The equivalence $\rho$ can be refined as follows. Let $(F,+,\times,0,1)$ be an idempotent semifield. Let $u\in F$. We say that $u$ is a strong unit of $F$ is $u\not=0$ and for every $x\in F,x\not=0$, there is $n$ such that $x^n\leq u$. Again, by removing the zero of the semifield and by composing with functor $\Gamma$, we get the following:

\begin{theorem}\label{thm:rumpmv} (\cite{DNR2}) There is an equivalence between idempotent semifields with a strong unit and MV-algebras.
\end{theorem}

By composing the functors $\Delta$ and $\rho$ we have another result concerning perfect MV-algebras:

\begin{theorem}\label{thm:perf} There is an equivalence between the categories of idempotent semifields  and perfect MV-algebras.
\end{theorem}

The previous theorems support the claim that tropical geometry can be made not only on semifields, $\emph{but even on MV-algebras}$.
That is we get a perspective for a tropical geometry of algebraic models of
{\L}ukasiewicz logics.

Theorems \ref{thm:rumpmv} and \ref{thm:perf} are the first steps of a comparison between tropical geometry and geometry over MV-algebras.

\section{Rump tropicalization and its generalization via $V(C)$ algebras: the functor $\theta$}\label{sect:rump}

In \cite{R} the author
established a category equivalence between  abelian $\ell$-groups and idempotent semifields.
We say that the {\em tropicalization} of a category (in the sense of Rump) is an equivalence functor over idepotent semifields.

Consider now the variety $V(C)$ generated by Chang's MV-algebra.

\medskip


\begin{definition}
Given any MV-algebra $A\in V(C)$, we let $$\theta(A)=\{x\in A|x\geq 2x^2\},$$
$$\theta^*(A)=\{x\in A|x\leq 2x^2\},$$
 where $x^2=x\odot x$ and $2x=x\oplus x$.
\end{definition}
For example, for $C=\Gamma(\ZZ \lex \ZZ,(1,0)$ we have
$\theta(C)=\{(0,n)\mid n \in \NN\} \cup \{(1,0)\}$.
Since $\theta(A)$ is the zeroset of an MV-polynomial, that is an algebraic variety of the $V(C)$, it is invariant with respect to isomorphic MV-algebras. In the following we collect some properties of $\theta$ and $\theta^*$:

\begin{proposition} For every $A\in V(C)$, $x\in\theta^*(A)$ if and only if $\neg x\in \theta(A)$. Hence $\theta(A)$ is closed under negation if and only if $\theta(A)=\theta^*(A)$.
\end{proposition}

\begin{proof}
$x\in\theta^*(A)$ iff $x\leq 2x^2$, and applying the negation, $\neg x\geq \neg(2x^2)$ iff $\neg x\geq 2(\neg x)^2)$ iff $\neg x\in \theta(A)$.

\end{proof}

\begin{proposition}
For every $A\in V(C)$, $B(A)\subseteq \theta(A)$ (in particular, $0,1\in\theta(A)$) and
 $Rad(A)\subseteq \theta(A)$. Further,  $co-Rad(A)\cap \theta(A)=\{1\}$ and, for every $x \in A$,
$x \in \theta(A)$ if and only if $x$ is a sum of a boolean and an infinitesimal.
\end{proposition}





We note:

\begin{lemma} If $A\in V(C)$, then $\theta(A)$ is closed under $\oplus,\odot,\wedge,\vee$.
\end{lemma}

\begin{proof} Suppose $x\geq 2x^2$ and $y\geq 2y^2$. Since every algebra in $V(C)$ is subdirect product of linearly ordered $V(C)$-algebras $L_i$ for $i$ in an index set $I$,   we can consider the $i$-th component of $x,y$ and we have $x_i\geq 2x_i^2$ and $y_i\geq 2y_i^2$. Then $x_i$ and $y_i$ can only be infinitesimal or $1$. So all their combinations under $\oplus,\odot,\vee,\wedge$ are also infinitesimals or $1$ and they belong to $\theta(L_i)$.
\end{proof}

In general $\theta(A)$ is not closed under negation. Indeed, suppose $\theta(A)$ is closed under negation. Then infinitesimals cannot exist in $A$, otherwise infinitesimals would exist in $\theta(A)$ and, by negation, coinfinitesimals would exist in $\theta(A)$. But coinfinitesimals never satisfy $x\geq 2x^2$ unless $x=1$. So $A$ is a semisimple MV-algebra in $V(C)$, hence $A$ is a Boolean algebra and $A=\theta(A)$, in particular $\theta(A)$ is an MV algebra.

\begin{definition}
 An $\ell$-Bisemiring an algebra $(S, \wedge, \vee, \odot, \oplus, 0, 1)$ such that:

\begin{enumerate}
\item $(S,\wedge,\vee,0,1)$ is a bounded distributive lattice;
\item $(S,\wedge,0,1,\oplus)$ is an idempotent commutative semiring;
\item $(S,\vee,0,1,\odot)$ is an idempotent commutative semiring.
\end{enumerate}
\end{definition}
 Clearly, every MV-algebra is an $\ell$-Bisemiring.

 We denote by $\ell BS$ the category of \emph{$\ell$-Bisemirings} and their morphisms.

 \begin{theorem} For every $A\in V(C)$, the structure $\theta(A)$ belongs to $\ell BS$ and the map sending $A$ to $\theta(A)$ is  functor from $V(C)$ to $\ell BS$. The same holds for $\theta^*$.
 \end{theorem}

 \begin{proof}
 In fact,  if $h:A\to B$ is a homomorphism of MV-algebras in $V(C)$, then $h$ sends $\theta(A)$ in $\theta(B)$ because $\theta(A)$ is the zeroset of an MV-algebraic polynomial, namely, $2x^2\ominus x$. So we can associate to $h$ the map $\theta(h):\theta(A)\to\theta(B)$ obtained by restricting $h$.

 More precisely, the functor $\theta$ goes from the category $V(C)$ seen as a category of MV-algebras, to a category of structures $(S,\wedge,\vee,\oplus,\odot)$ with four binary operations (the first two are distributive lattice operations).  Note that any homomorphism of MV-algebras  preserves the four operators.

 A similar proof holds for $\theta^*$.

 \end{proof}

 In a sense, negation gives a duality between the functors $\theta$ and $\theta^*$.

 For $A\in V(C)$ we note that an element $x\in\theta(A)$ is boolean in $A$ if and only if $x^2=x$, and $x$ is infinitesimal in $A$ if and only if $x^2=0$. This means that booleans and infinitesimals are definable internally in $\theta(A)$, since multiplication is available in $\theta(A)$.
We begin with a lemma which allows us to put together a boolean algebra and a perfect algebra.

 \begin{lemma} Let $B$ be a boolean algebra and $P$ a perfect MV-algebra. There is an MV-algebra $A\in V(C)$ such that $B(A)=B$ and $Rad(A)=Rad(P)$.
 \end{lemma}

 \begin{proof} We perform the following detour. Let $B'$ be a boolean algebra such that $B\subseteq B'\times\{0,1\}$ ($B'$ exists by Stone Theorem on Boolean algebras). Consider $A=\{(b,p)\in B'\times P|(b,p\ mod\ Rad(P))\in B\}$.

 Note that $A\in V(C)$. Moreover,  $B(A)=B$ and $Rad(A)$ is the set of the elements $(0,p) $ with $p\in Rad(P)$. So $Rad(A)$ is isomorphic to $Rad(P)$.

 \end{proof}

 Now we have the desired characterization of $\theta$ of MV-algebras in $V(C)$:

 \begin{lemma}\label{lemma:nablageom} Let $S$ be an $\ell$-Bisemiring.  Then $S=\theta(A)$ for some $A\in V(C)$ if and only if:
\begin{itemize} \item $(Inf(S),\oplus,\odot,\wedge,\vee)$ is the radical of an MV-algebra, where $Inf(S)$ is the set of $x\in S$ such that $x^2=0$ (infinitesimal elements).
\item $(BOOL(S),\oplus,\odot)$ is a Boolean algebra, where $\vee=\oplus$ and $\wedge=\odot$ and $Bool(S)$ is the set of $x$ such that $x^2=x$ (Boolean elements).

\end{itemize}
\end{lemma}

In particular, if $S=\theta(A)$ for a perfect MV-algebra $A$, then $BOOL(S)=\{0,1\}$.

\subsection{Tropicalization of Perfect MV-algebras}\label{sect:perf}

In this section we continue to describe the algebraic aspects of the algebraic varieties $\theta(A)$ in the cases where $A$ is a Perfect MV-algebra. Perfect MV-algebras are a full subcategory of $V(C)$. We remark that they Perfect MV-algebras do not form an equational class.

We explore the role of $\theta$ in composition with other functors acting in tropical geometry, and more, its  role in moving from $V(C)$-algebraic varieties into tropical algebraic varieties.
\begin{definition}
Let $TOP$ be the category of positive cones of $\ell$-groups with a top, where morphisms preserve the positive cone and the top,
$PERF$ be the category of perfect MV-algebras and
$\ell \cal G$ be the category of abelian $\ell$-groups.
\end{definition}
For any perfect MV-algebra $P$, $\theta(P) = Rad(P)\cup \{1\}$ and $Rad(P)$ is the positive cone of the $\ell$-group $\Delta^{-1}(P)$. On the other hand,  the $\Delta$ functor maps every  abelian $\ell$-group $G$ in a perfect MV-algebra $\Delta(G)$.
We can extend such correspondences to equivalences of categories:
\begin{theorem} $\theta$ is an equivalence functor between $PERF$ and $TOP$.
$\theta\circ\Delta$ is a categorial equivalence from $\ell \cal G$ to $TOP$.
\end{theorem}

If we specialize from $V(C)$ to $Perf_\QQ$ we obtain more:

\begin{corollary} The map $\Theta_\QQ$ sending $A\in Perf_\QQ$ to $\Theta(A)$ is an equivalence of categories  from $Perf_\QQ$ to a full subcategory of $C\ell BS$.
\end{corollary}

\begin{proof} Let $\mathcal{R}$ be the full subcategory of $C\ell BS$ consisting of $\Theta(A)$ for $A\in Perf_\QQ$. $\Theta_\QQ$ is a full and faithful functor from
$Perf_\QQ$ to $\mathcal{R}$.
\end{proof}

If $G$ is an abelian $\ell$-group, we define its tropicalization
$Trop(G)$ as the additively idempotent semifield obtained by adding a point $-\infty$ to $G$. The tropicalization can be considered as a functor from $\ell \cal G$ to the category of idempotent semirings. The inverse functor takes an additively idempotent semifield $S$ and maps it to the abelian $\ell$-group $Detrop(S)$ obtained by deleting the zero of $S$.

So, given an additively idempotent semifield $S$, we can define the functor $\mathbf{F}$ by  a triple composition of functors

$$\mathbf{F}(S)=\theta(\Delta(Detrop(S))),$$

\begin{theorem} $\mathbf{F}$ is an equivalence from the category of additively idempotent semifields to $TOP$.
\end{theorem}

So the functor $\mathbf{F}$ is  an equivalence, and gives a direct link between tropical structures ( additively idempotent semifields) and our $\theta$ structures over Perfect MV-algebras. The construction is functorial.

So, the above theorem build up a \emph{functorial} \emph{tropicalization} of algebraic structures related with {\L}ukasiewicz logic. Indeed, we see the elements of $TOP$ as $\emph{tropicalizations}$ of perfect MV-algebras. It remains to tropicalize in a reasonable sense more general MV-algebras.

\section{MV-algebras and the topos $\widehat{\NN^\times}$}

We are inspired by \cite{CC} where the authors apply topos theory and tropical geometry for a possible road to Riemann Hypothesis. The hope is that, as schemes opened the way for Fermat's theorem, topoi open the way for Riemann hypothesis.

We recall the reader that a {\em topos} is a mathematical structure that generalizes the concept of a space, extending the ideas of set theory and logic into a categorical framework. Originating from Grothendieck's work in algebraic geometry, topoi have found applications in diverse fields such as logic, computer science, and topology. At its core, a topos can be viewed as a category that behaves like the category of sets, satisfying specific axioms that make it a generalization of a ``universe of discourse."

One of the simplest examples of a topos is a {\em presheaf category}, denoted as \(\text{Set}^{\mathcal{C}^{\text{op}}}\), where \(\mathcal{C}\) is a small category. Objects in this category are functors from \(\mathcal{C}^{\text{op}}\) to the category of sets, and morphisms are natural transformations between these functors. A presheaf hence assigns sets to objects in \(\mathcal{C}\) and specifies how these sets transform along morphisms of \(\mathcal{C}\).

A key feature of topoi is the concept of {\em points}, which abstractly generalize the idea of geometric points in a space. A point of a topos \(\mathcal{E}\) is defined as a geometric morphism \(\mathsf{Set} \to \mathcal{E}\), where \(\mathsf{Set}\) is the category of sets. Such a morphism consists of an adjoint pair of functors that preserve the essential structure of the topos. Points play a significant role in connecting the categorical structure of a topos to more classical notions of geometry and logic.

In summary, topoi unify geometric and logical perspectives, presheaf categories provide an intuitive entry point into their study, and points of a topos link them to tangible, classical ideas. Together, they form a rich and versatile framework for exploring abstract and applied mathematics.

In \cite{CC} we found the crucial construction from which we start our investigation:
\begin{definition}
We denote by $\NN^\times$ the monoid of integers with multiplication and by
$\widehat{\NN^\times}$ its presheaf category of the functors from $\NN^\times$, seen as a category, to $Set$.
\end{definition}

The points of the topos $\widehat{\NN^\times}$ form a category which is studied in \cite{CC}.

We aim to give a MV-algebraic characterization of the points of the presheaf topos $\widehat{\NN^\times}$ and to investigate the logical consequences of this characterization.


Recall that by $\Delta$ we denote the categorical equivalence between abelian $\ell$-groups and perfect
 MV-algebras, see (\cite {DL94}). Let us call $\Delta^\QQ$ the restriction of the functor $\Delta$ to $Ab^\QQ$, i.e., the category of abelian $\ell$-subgroups of $\QQ$. Then we set $Perf^\QQ:=\Delta^\QQ(Ab^\QQ)$. Likewise we can consider  the restriction $\theta^\QQ$ of $\theta$ to the range of $\Delta^\QQ$ and $TOP^\QQ$ the image of $\Delta^\QQ$.

We start by investigating the first order theory of the non trivial subgroups of the rationals as ordered groups. Then via the $\Delta^\QQ$ functor we can transfer the obtained results to the first order theory of $Perf^\QQ$ and to points of the topos $\widehat{\NN^\times}$.

We can consider the logic $T_\QQ$ which is the first order theory of all ordered subgroups of $\QQ$ and $T_{\Delta(\QQ)}$ the first order theory of $Perf^\QQ$. Then we have:

\begin{theorem}\label{thm:tq}  The first order theory $T_\QQ$  is recursively axiomatizable.
\end{theorem}

\begin{proof} In \cite{RZ} we find a complete characterization of a class of ordered abelian groups, called regularly ordered, which includes archimedean groups. We take from \cite{RZ} the definitions of regularly ordered group and $Gp$.

If $G$ is a group, and $p$ is a prime number, $Gp$ is the largest finite number of elements of $G$ pairwise non-congruent modulo $p$, or $\infty$ if this number does not exist.

A group is discretely ordered if is has a smallest positive element.

A group is regularly discrete if $G$ is discretely ordered and such that $Gp=p$ for every prime $p$.

$G$ is regularly dense if, for any positive integer $p$ and any elements $a<b$,  there is $x$ between $a$ and $b$ which is divisible by $p$.

$G$ is regularly ordered if  it is regularly discrete or regularly dense.

The axioms of $T_\QQ$ are:  1) $Gp\leq p$ for every prime $p$. 2) $G$ is regularly ordered.

Clearly 1) is expressible by (a recursive set of) first order axioms. For 2),
note that $G$ is regularly ordered if it has an atom and for every pair of integers, $p,q$, either $G$ has an atom and $p$ is prime and is $p$-regularly discrete, or $G$ is $q$-regularly dense. So, also 2) is expressible by (a recursive set of) first order axioms.

In fact, by \cite{RZ}, every regularly ordered group with $Gp\leq p$ for every prime $p$ is elementarily equivalent to a subgroup of $\QQ$.

Conversely, if $G$ verifies the theory of all the subgroups of $\QQ$, then it is regularly ordered, and verifies $Gp\leq p$ for every $p$.

\end{proof}

{\em Geometric logic} is a fragment of infinitary first order  logic that focuses on formulas and reasoning preserved under colimits. Unlike classical logic, which emphasizes arbitrary conjunctions and negations, geometric logic emphasizes finite conjunctions, arbitrary disjunctions, and existential quantifiers. This focus makes geometric logic particularly well-suited to categorical settings and topoi.

While geometric logic primarily allows existential quantifiers, universal quantifiers can be used if the formula being universally quantified can be equivalently rewritten using finite conjunctions and negation of atomic formulas, which geometric logic permits implicitly.

A geometric theory is a collection of axioms expressed in geometric logic, and its classifying topos is a category that serves as a universal semantic model for the theory. The connection between syntactic (logical) and semantic (categorical) perspectives exemplifies the deep interplay between geometric logic and topoi.

Geometric logic turns out to be the right logic for topos theory. Every geometric theory has a classifying topos, and when two theories have the same topos (that is they are Morita equivalent) they have many properties in common, despite they may look very different. For instance, the theories of MV-algebras and abelian $\ell$-groups with strong unit are Morita equivalent, see \cite{CR15}, and the theories of perfect MV-algebras and abelian $\ell$-groups are Morita equivalent, see
\cite{CR16}. So, it seems natural to look for a geometric logic of $\widehat{\NN^\times}$.

\begin{theorem} The class of ordered subgroups of $\QQ$ is geometric in the language of ordered groups.
\end{theorem}
\begin{proof} $G$ is a subgroup of  $\QQ$ if and only if $G$ is a group, $G$ is torsion free, and for every $0 \neq x,y\in G$ there are $m,n\in \NN$ such that $mx=ny$. Furthermore, denoting the order of $G$ by $<_{G}$, the following conditions hold:
$\forall x$ $\forall y$ $\exists n, m\in \mathbb{Z}$ with
\begin{itemize}
\item  $ny = mx$
\item  (($na=mx$) and ($qb=px$)) implies (($a<_{G} b$ iff $mq<np$)), with $p,q\in \QQ$.
\end{itemize}

 All these properties are geometric.
\end{proof}

\begin{theorem} The class $Perf^{\QQ}$ is geometric in the language of $MV$-algebras.
\end{theorem}
\begin{proof}
Let $A\in Perf^{\QQ}$. Then $(\Delta(A))^{+}\subseteq \QQ$, that is equivalent to the axiom:
\vskip 15pt

$\forall x$ $\forall y$ $\exists n$ $\exists m$  such that: $m,n\in {\mathbb Z}^{+}$

\begin{itemize}
\item  $ny = mx$

\item  (($na=mx$) and ($qb=px$)) implies (($a<_{G} b$ iff $mq<np$)), with $p,q\in \QQ$.

\end{itemize}

 All these properties are geometric.
\end{proof}

Finally, using some of above results of this section, we obtain algebraic and model theoretic properties of the points of the topos $\widehat{\NN^\times}$. We start with

\begin{theorem} (\cite{CC}, theorem 2.1)
The category $\bf{Pt}$$(\widehat{\NN^\times})$ of the points of the topos $\widehat{\NN^\times}$ is equivalent to the category $Ab^\QQ$ of Abelian $\ell$-groups included in $\QQ$ where only increasing group homomorphisms are considered as morphisms.
\end{theorem}
By composing with the functor $\Delta$ we obtain:
\begin{corollary}  The category $\bf{Pt}$$(\widehat{\NN^\times})$ of the points of the topos $\widehat{\NN^\times}$ is equivalent to the category $Perf^\QQ$.
\end{corollary}

Let us introduce a further definition.
\begin{definition}
If $\mathbf{p}\in Pt(\NN^\times))$, let $Pt_\QQ(\mathbf{p})$ be the corresponding subgroup of $\QQ$. The map $pt_\QQ$ is a functor.  So for every $\mathbf{p}\in Pt(\NN^\times)$ we define
$$\Theta_{pt}(\mathbf{p})=\Theta_\QQ ((\Delta^\QQ (Pt_\QQ(\mathbf{p})))$$
and $\Theta_{pt}$ is a functor from $Pt(\widehat{\NN^\times})$ to $\ell BS$.
\end{definition}

So we have

\begin{corollary} There is a natural functor from the category $Pt(\widehat{\NN^\times})$ to the category of $\ell$-Bisemirings.
 \end{corollary}

Given the geometric characterization, via a categorical equivalence, of algebras from $Perf^\QQ$ as points of the natural topos $\widehat{\NN^\times}$, it would be interesting to state results about the first order theory and geometric theory of $Perf^\QQ$ to get logical information about the points of $\widehat{\NN^\times}$.

Indeed, we have a threefold interaction among the geometric side, points of the topos
$\widehat{\NN^\times}$, the algebraic side of ordered subgroups of $\QQ$ and the logical side of $Perf^\QQ$, models of an extension of {\L}ukasiewicz logic.






\subsection{Flatness and $Pt(\widehat{\NN^\times})$}

In \cite{CC} theorem 2.1 holds because the points of $\widehat{\NN^\times}$ correspond to flat functors from $\NN^\times$ to Set, we can call them flat actions of $\NN^\times$ and we denote by $Flact(\NN^\times)$ their category.
\begin{definition}
Recall that a functor $F:\NN^\times\to Set$ is flat if
\begin{enumerate}
\item $F(*)=X\not=\emptyset$, where $*$ is the unique object of the category $\NN^\times$
\item given $y,z\in X$ there are $m,n\in \NN$ and $w\in X$ such that $mw=y$ and $nw=z$
\item if $m,n\in \NN$ and $y\in X$ such that $my=ny=w$, then there is $p\in \NN$ and $z\in X$ such that $mp=np$ and $pz=y$.
\end{enumerate}
\end{definition}
Note that all flatness axioms above are expressed in the geometric logic in the sense of topos theory. In particular,  flat functors are closed under colimits.

 In lemma 2.2 of \cite{CC} we find:

\begin{lemma} There is a categorical equivalence between the category $Flact(X)$  and the category of ordered subgroups of the rationals.
\end{lemma}

\begin{proof} Given an action on $X$, the idea is to define a sum on $X$ by $kz+k'z=(k+k')z$. This is well defined since the action is flat, and gives a semigroup which is the positive part of a subgroup of the rationals.  Conversely, the Frobenius action on the positive part of a subgroup of the rationals is flat.
\end{proof}

We remark that the point $\mathbf{p_\QQ}\in \widehat{\NN^\times}$ corresponding to $\QQ$ is the exponentiation functor $F_\QQ:\NN\times X\to X$, with $X=\QQ^{+}$ and $q\in \QQ^{+}$, $F_{\QQ}(n,q)\rightarrow q^{n}$.

Let $L_{Pt}$ be the language with a sequence of functional symbols $\{F_{n}:\NN\times X \rightarrow X\}$. Then we have:

\begin{theorem}
The class $Pt(\widehat{\NN^\times})$ is geometric over the language $L_{Pt}$.
\end{theorem}

\begin{corollary}
Every point $\mathbf{p}\in Pt(\widehat{\NN^\times})$ is an elementary substructure of $\mathbf{p_{\QQ}}$.

\end{corollary}

\section{Conclusions}\label{sect:concl}

As stated in the Introduction, the purpose of this article is to illustrate the interactions between algebraic structures such as Semirings, Tropical structures and {\L}ukasiewicz logic.

This perspective has led to a generalisation of the concept of Tropicalisation of algebraic structures in a form that opened up an unexpected interpretation of an extension of {\L}ukasiewicz logic. Idempotent semifields are the basic algebraic structures needed to built a tropical version of algebraic geometry: here we showed that if we focus on particular on semifields arising from subclasses of MV-algebras, we can connect the tropical geometry with the rich world of logic, through a topos prespective.

Considering the functorial version of the tropicalization of an algebraic structure, the above interation, among others, allows to show that
from the geometric side, the category of points of the topos $\widehat{\NN^\times}$ results equivalent to several algebraic categories.

As possible continuation of this work, the topics in Section 4 can be extended to find computability and complexity results of first order theories of MV-algebras and of points of the topos $\NN^\times$.
Further, the connections between the logic of perfect MV-algebras, topos and categories arising from idempotent semirings still need a deep investigation.

\end{document}